\documentclass{article}
\usepackage[margin=1.95cm]{geometry}
\usepackage[hidelinks]{hyperref}
\usepackage{amsmath,amsthm,amsfonts,amssymb,amscd,mathrsfs}
\usepackage[shortlabels]{enumitem}
\usepackage[T1]{fontenc}
\usepackage{xcolor}

\newtheorem{theorem}{Theorem}[section]

\newtheorem{proposition}[theorem]{Proposition}

\newtheorem{example}[theorem]{Example}
\theoremstyle{definition}
\newtheorem{assumption}[theorem]{Assumption}
\theoremstyle{remark}
\newtheorem{remark}[theorem]{Remark}

\newcommand{\E}{{\mathbb E}}

\title{$L^p_\mu$-Normed Spaces of Bounded Multilinear Operators}

\author{ Filippo de Feo\footnote{Institut für Mathematik, Technische Universität Berlin. Email: defeo@math.tu-berlin.de}}
\begin{document}

\maketitle

\begin{abstract}
We study the spaces of bounded multilinear operators $\mathcal L^i(X,Y)$ endowed with $L^p_\mu$-norms, where    $X,Y$ are normed spaces and $\mu$ is a measure with  $p$-moments. These are norms  without taking a quotient, under suitable support conditions.  {The topology induced coincides with the strong operator and compact-open topologies on operator-norm-bounded subsets.}  Furthermore, we identify the  completion of these spaces with spaces of Hilbert--Schmidt operators on the  reproducing kernel Hilbert space (RKHS) for $p=2$; with iterated $\gamma$-radonifying multilinear operators on the  RKHS  in the Gaussian case, for general $p$.
\end{abstract}

\tableofcontents

\section{Introduction}\label{sec:intro}
\paragraph{Contributions.}
Throughout the paper, $X,Y\neq\{0\}$ are normed spaces and $i\geq1$. We impose the following standing assumption unless otherwise stated.

\begin{assumption}[Mixed moment condition]\label{ass:moments}
Let $1\leq p<\infty$ and $\mu$ be a positive Radon measure on $X^i$ satisfying $\int_{X^i}\prod_{t=1}^i|h_t|_X^p\,d\mu(h_1,\ldots,h_i)<\infty$.
\end{assumption}
 We consider spaces of bounded multilinear operators $\mathcal L^i(X,Y)$ equipped with the following $L^p_\mu$-norm\footnote{Not to be confused with the integral norm of (Grothendieck-)integral multilinear operators, see e.g.\ \cite{villanueva2003}.}:
\begin{equation}\label{norm:Lp_operators}
\|L\|_{L^p_\mu(X^i,Y)}:=\left(\int_{X^i} |L(h_1,...,h_i)|^p_Yd\mu(h_1,\ldots,h_i)\right)^{1/p},\quad L\in \mathcal L^i(X,Y).
\end{equation}
The main contribution of the manuscript is to develop a systematic theory of the resulting normed spaces and their relations to  classical topologies on spaces $\mathcal L^i(X,Y)$ available  in the literature:
\begin{itemize}
\item \emph{Characterization.} We characterize conditions for which $X$ and $\mu$ lead to a proper norm without quotient. In contrast to Bochner $L^p$-type norms, which are only defined on equivalence classes \cite{hytonen2016analysis}, $\|\cdot\|_{L^p_\mu(X^i,Y)}$ is a norm on $\mathcal L^i(X,Y)$ without quotient, for instance, if $\operatorname{supp}\mu\supset(S^{X,1})^i$, by linearity and continuity (Proposition~\ref{lem:seminorm_norm}).  Moreover, a Radon measure $\mu$ for which $\|\cdot\|_{L^p_\mu(X^i,Y)}$ is a norm exists if and only if $X$ is separable (Proposition~\ref{prop:Lp_norm_separable}).
\item \emph{Comparison with the operator norm.} For every such $\mu$, $\|\cdot\|_{L^p_\mu(X^i,Y)}$ and $\|\cdot\|_{\mathcal L^i(X,Y)}$ are equivalent if and only if $X$ is finite dimensional, and, if $Y$ is a Banach space, $\big(\mathcal L^i(X,Y),\|\cdot\|_{L^p_\mu(X^i,Y)}\big)$ is complete if and only if $X$ is finite dimensional (Proposition~\ref{prop:finite-dim}). Hence, if $X$ is infinite dimensional, these norms are strictly weaker than the operator norm and, for Banach $Y$, never complete.
\item \emph{Comparison with  SOT, COT, WOT.} On operator-norm-bounded subsets of $\mathcal L^i(X,Y)$, all such $L^p_\mu$-norms, $1\leq p<\infty$, induce the same topology as the strong operator topology (SOT) and the compact-open topology (COT), independently of the measure and the exponent (Proposition~\ref{prop:bounded_sets}). If $Y$ is finite dimensional, these topologies also coincide with the weak operator topology (WOT).
\item \emph{Explicit completions.} For $p=2$ and $\mu=\otimes_{j=1}^i\nu$, with $\nu$ a centered probability measure with finite second moment and $Y$ a Hilbert space, or $\nu$ a centered Gaussian measure and $Y$ a Banach space (see Examples~\ref{ex:HS_norms_RKHS} and~\ref{ex:gaussian_measures_radonifying}), we identify the completion isometrically with the space of Hilbert--Schmidt, respectively iterated $\gamma$-radonifying, multilinear operators on the reproducing kernel Hilbert space of $\nu$ (Theorem~\ref{thm:completion}).
\end{itemize}
Motivations for studying these norms come, for instance, from Derivative-Informed Operator Learning, i.e.\ the learning of nonlinear operators $F:X\to Y$ together with their derivatives $D^iF:X\to\mathcal L^i(X,Y)$, where one may want to measure errors of derivatives in integral norms with respect to sampling probability measures \cite{defeo2026universal} (see also \cite{cohen-defeo-hebner-sirignano})\footnote{Operator norms are generally too strong for this purpose: classical architectures fail to universally approximate the derivatives $D^iF(x)$ in these norms.}. 
\paragraph{Outline.}
The manuscript is organized as follows. In Section~\ref{subsec:norm_Lp-multilinear}, we study the main properties of $\|\cdot\|_{L^p_\mu(X^i,Y)}$. In Section~\ref{sec:examples}, we  discuss relations with classical examples: discrete measures, and Hilbert--Schmidt and $\gamma$-radonifying norms. Appendix~\ref{app:basics} collects notations.

\section{Properties}\label{subsec:norm_Lp-multilinear}
In this section, we study when $\|\cdot\|_{L^p_\mu(X^i,Y)}$ is a norm on $\mathcal L^i(X,Y)$, how it compares with $\|\cdot\|_{\mathcal L^i(X,Y)}$, and when $\mathcal L^i(X,Y)$ is complete with respect to it, for $1\leq p<\infty$.

The following proposition gives a sufficient condition for $\|\cdot\|_{L^p_\mu(X^i,Y)}$ to be a norm on $\mathcal L^i(X,Y)$; moreover, $\|\cdot\|_{L^p_\mu(X^i,Y)}$ is always weaker than $\|\cdot\|_{\mathcal L^i(X,Y)}$.
\begin{proposition}\label{lem:seminorm_norm}
Let Assumption~\ref{ass:moments} hold. Then the map $\|\cdot\|_{L^p_\mu(X^i,Y)}$ defined in \eqref{norm:Lp_operators} is a seminorm on $\mathcal L^i(X;Y)$, and it is weaker than $\|\cdot\|_{\mathcal L^i(X,Y)}$, since
\begin{equation}\label{eq:normr_bounded_opnorm}
    \|L\|_{L^p_\mu(X^i,Y)}\leq \|L\|_{\mathcal L^i(X,Y)}\left(\int_{X^i} \prod_{t=1}^i |h_t|_X^pd\mu(h_1,\ldots,h_i)\right)^{1/p}.
\end{equation}
Moreover, if {$\operatorname{supp}\mu\supset {(S^{X,1})^i}$, then it is a norm on $\mathcal L^i(X;Y)$; in particular, no quotient with respect to equivalence relations is needed, unlike for standard $L_\mu^p$-spaces.}
\end{proposition}
\begin{proof}
The bound \eqref{eq:normr_bounded_opnorm} follows from Assumption~\ref{ass:moments} and the inequality (see Proposition~\ref{lemma:multilinear}):
$$|L(h_1,\ldots,h_i)|_Y\leq\|L\|_{\mathcal L^i(X,Y)}\prod_{t=1}^i|h_t|_X.$$

{To check that $ \|\cdot\|_{L^p_\mu(X^i,Y)}$ is a norm when $\operatorname{supp}\mu\supset {(S^{X,1})^i}$, it is enough to check that $\|L\|_{L^p_\mu(X^i,Y)}=0$ implies $L=0$. Indeed,  assume $\|L\|_{L^p_\mu(X^i,Y)}{=}0$. 
Denote $$g(h_1,\ldots,h_i):=|L(h_1,\ldots,h_i)|_Y^p.$$ Since $L\in \mathcal L^i(X,Y)$, $g$ is continuous and nonnegative. 
If $g(h_{0,1},\ldots,h_{0,i})>0$  at some $(h_{0,1},\ldots,h_{0,i}) \in \operatorname{supp} \mu$, continuity     would give an open neighborhood $U\subset X^i$ such that $(h_{0,1},\ldots,h_{0,i}) \in U$ and $g>g\left(h_{0,1},\ldots,h_{0,i}\right) / 2$ on $U$; by the definition of support, $\mu(U)>0$, from which $\int_{X^i} g d \mu\geq \int_U g d \mu \geq \mu(U) g\left(h_{0,1},\ldots,h_{0,i}\right) / 2$, contradicting $\int_{X^i} g d \mu=0$. Hence $$g=0,\quad \textit{on } \operatorname{supp}\mu\supset {(S^{X,1})^i},\quad \textit{i.e. }\|L\|_{\mathcal L^i(X,Y)}=0$$ Then $L=0$.} 
\end{proof}

\begin{remark}\label{rem:target_nondegeneracy}
We will use the following observation: if $\|\cdot\|_{L^p_\mu(X^i,Y)}$ is a norm on $\mathcal L^i(X,Y)$, then $\|\cdot\|_{L^p_\mu(X^i,\mathbb R)}$ is a norm on $\mathcal L^i(X,\mathbb R)$. Indeed, assume that $\|z\|_{L^p_\mu(X^i,\mathbb R)}=0$ and
fix $y_0\in Y\setminus\{0\}$. For every $z\in\mathcal L^i(X,\mathbb R)$, we have $$\|z\,y_0\|_{L^p_\mu(X^i,Y)}=|y_0|_Y\|z\|_{L^p_\mu(X^i,\mathbb R)}=0,$$ so that   $z\,y_0=0$, and hence $z=0$.
\end{remark}

\begin{remark}\label{rem:moment_condition}
The moment condition in Assumption~\ref{ass:moments} ensures, through
\eqref{eq:normr_bounded_opnorm}, that
$\|\cdot\|_{L^p_\mu(X^i,Y)}$ is finite on $\mathcal L^i(X,Y)$,
but it is not necessary for this in general. For instance, for $i=1$ and $Y=\mathbb R$, finiteness on $X^*$ is precisely
the existence of weak $p$-th moments, i.e.\
$\int_X|x^*(h)|^p\,d\mu(h)<\infty$ for every $x^*\in X^*$.
In infinite-dimensional  spaces, this condition is strictly
weaker than $\int_X|h|_X^p\,d\mu(h)<\infty$, see
\cite[Section II.2.1, Theorem 2.1]{vakhania-tarialadze}.
\end{remark}
We first characterize the spaces $X$ for which some Radon measure makes the seminorm in \eqref{norm:Lp_operators} a norm.
{\begin{proposition}\label{prop:Lp_norm_separable}
$X$ is separable if and only if there exists a 
measure $\mu$ satisfying Assumption~\ref{ass:moments} such    that $\|\cdot\|_{L^p_{\mu}(X^i,Y)}$ is a norm on $\mathcal L^i(X,Y)$. Consequently, if $X$ is nonseparable, then $\|\cdot\|_{L^p_{\mu}(X^i,Y)}$ is a proper seminorm for every measure $\mu$ satisfying Assumption~\ref{ass:moments}.
\end{proposition}
\begin{proof}
($\Rightarrow$) Let $X$ be separable and let $\{x_n\}_{n\in\mathbb N}$ be dense in $S^{X,1}$. Set $\nu:=\sum_{n\in\mathbb N}2^{-n}\delta_{x_n}$ and $\mu:=\otimes_{j=1}^i\nu$, a Radon probability measure on $X^i$ with $\int_{X^i}\prod_{t=1}^i|h_t|_X^p\,d\mu=1$ and we have\footnote{Each point \((x_{n_1},\ldots,x_{n_i})\) belongs to \(\operatorname{supp}\mu\), since every neighborhood of it has positive \(\mu\)-measure. As the support is closed and these points are dense in \((S^{X,1})^i\), we obtain \((S^{X,1})^i\subseteq\operatorname{supp}\mu\). Conversely, \(\mu\) is concentrated on the closed set \((S^{X,1})^i\), so \(\operatorname{supp}\mu\subseteq(S^{X,1})^i\). Hence \(\operatorname{supp}\mu=(S^{X,1})^i\).} $\operatorname{supp}\mu=(S^{X,1})^i$, so that $\|\cdot\|_{L^p_{\mu}(X^i,Y)}$ is a norm by Proposition~\ref{lem:seminorm_norm}. 

($\Leftarrow$) Let Assumption~\ref{ass:moments} hold. Assume that $\|\cdot\|_{L^p_{\mu}(X^i,Y)}$ is a norm.  We can assume that $\mu$ is finite in view of the following. Set $w(h_1,\ldots,h_i):=\prod_{t=1}^i|h_t|_X^p$ and $d\nu:=\frac{w}{1+w}\,d\mu$. Then $\nu$ is finite Radon and satisfies the same moment condition, since $\nu(X^i)+\int_{X^i}w\,d\nu=\int_{X^i}w\,d\mu<\infty$, and the corresponding seminorms have the same kernel. Thus, replacing $\mu$ by $\nu$, we may assume that $\mu$ is finite. 

Being    Radon, $\mu$ is concentrated    on a $\sigma$-compact, hence separable, set $K\subset X^i$ (i.e. $\mu(X^i\setminus K)=0$, with $K=\bigcup_{n\in\mathbb N}K_n$ for some compact sets $K_n\subset X^i$; by tightness, one can take $K_n$ such that $\mu(X^i\setminus K_n)<1/n$),     so that $$Z:=\overline{\operatorname{span}\,\pi_1(K)}, \qquad  \textit{where } \pi_1(h_1,\ldots,h_i):=h_1,$$ is a separable       closed subspace of $X$. We show that 
$$X=Z,$$
which, in turn, implies   that $X$ is separable,     concluding the proof.
 Indeed, assume by contradiction    that $Z\neq X$. Then, by the Hahn--Banach theorem, there exists $x^*\in X^*\setminus\{0\}$ vanishing on $Z$; fixing $y_0\in Y\setminus\{0\}$, the map\footnote{the product is needed: for $i\geq2$, the map $(h_1,\ldots,h_i)\mapsto x^*(h_1)\,y_0$ is constant in $h_2,\ldots,h_i$, hence not $i$-linear} $L(h_1,\ldots,h_i):=\big(\prod_{t=1}^i x^*(h_t)\big)\,y_0$ is    nonzero    and vanishes on $K$ (since, for $(h_1,\ldots,h_i)\in K$ we have $h_1\in\pi_1(K)\subset Z$, hence $x^*(h_1)=0$, which implies that $L$ vanishes on $K$) so that $$\|L\|_{L^p_{\mu}(X^i,Y)}^p=\int_K|L(h_1,\ldots,h_i)|_Y^p\,d\mu=0,$$ contradicting     the fact that $\|\cdot\|_{L^p_{\mu}(X^i,Y)}$ is a norm. Hence $X=Z$.
\end{proof}}

Whenever $\|\cdot\|_{L^p_\mu(X^i,Y)}$ is a norm, we compare it with the operator norm $\|\cdot\|_{\mathcal L^i(X,Y)}$.
\begin{proposition}\label{prop:finite-dim}
Let Assumption~\ref{ass:moments} hold. Assume that $\|\cdot\|_{L^p_\mu(X^i,Y)}$ is a norm on $\mathcal L^i(X,Y)$. Then:
\begin{enumerate}[(i)]
\item $\|\cdot\|_{L^p_\mu(X^i,Y)}$ and $\|\cdot\|_{\mathcal L^i(X,Y)}$ are equivalent if and only if $X$ is finite dimensional; if $X$ is infinite dimensional, then $\|\cdot\|_{L^p_\mu(X^i,Y)}$ induces a strictly weaker topology on $\mathcal L^i(X,Y)$ than $\|\cdot\|_{\mathcal L^i(X,Y)}$;
\item if $Y$ is a Banach space, then $\big(\mathcal L^i(X,Y),\|\cdot\|_{L^p_\mu(X^i,Y)}\big)$ is complete if and only if $X$ is finite dimensional.
\end{enumerate}
\end{proposition}
\begin{proof}
(i) By \eqref{eq:normr_bounded_opnorm}, the topology induced by $\|\cdot\|_{L^p_\mu(X^i,Y)}$ is weaker than the one induced by $\|\cdot\|_{\mathcal L^i(X,Y)}$.

Let $X$ be finite dimensional, and note that $\|\cdot\|_{L^p_\mu(X^i,\mathbb R)}$ is a norm on the finite-dimensional space $\mathcal L^i(X,\mathbb R)$ by Remark~\ref{rem:target_nondegeneracy}; hence, by equivalence of norms in this case, there exists $C>0$ such that $\|z\|_{\mathcal L^i(X,\mathbb R)}\leq C\|z\|_{L^p_\mu(X^i,\mathbb R)}$ for all $z\in\mathcal L^i(X,\mathbb R)$. Let $L\in\mathcal L^i(X,Y)$ and $h_1,\ldots,h_i\in X$ with $|h_1|_X=\cdots=|h_i|_X=1$. By the corollary of the Hahn--Banach theorem \cite[Corollary~1.4]{brezis2011}, there exists $y^*\in Y^*$ with $\|y^*\|_{Y^*}\leq1$ and $|y^*(L(h_1,\ldots,h_i))|=y:=|L(h_1,\ldots,h_i)|_Y$. Hence, applying the above equivalence of norms to $z=y^*\circ L$, we obtain
\begin{align*}
|L(h_1,\ldots,h_i)|_Y&=|(y^*\circ L)(h_1,\ldots,h_i)|\leq\|y^*\circ L\|_{\mathcal L^i(X,\mathbb R)}\leq C\|y^*\circ L\|_{L^p_\mu(X^i,\mathbb R)}\\
&\leq C\|y^*\|_{Y^*}\|L\|_{L^p_\mu(X^i,Y)}\leq C\|L\|_{L^p_\mu(X^i,Y)}.
\end{align*}
Taking the supremum over all such $h_1,\ldots,h_i$ gives $\|L\|_{\mathcal L^i(X,Y)}\leq C\|L\|_{L^p_\mu(X^i,Y)}$.

Let $X$ be infinite dimensional. We construct a sequence $(L_n)_{n\in\mathbb N}\subset\mathcal L^i(X,Y)$ such that $\|L_n\|_{\mathcal L^i(X,Y)}=1$ for all $n\in\mathbb N$ and $\|L_n\|_{L^p_\mu(X^i,Y)}\to0$, which shows that the two topologies differ (making the  one from $L^p_\mu$-norm strictly weaker). Indeed, by Proposition~\ref{prop:Lp_norm_separable}, $X$ is separable; let $\{x_k\}_{k\in\mathbb N}$ be dense in $X$. For $n\in\mathbb N$, $\operatorname{span}\{x_1,\ldots,x_n\}\neq X$, so, by the Hahn--Banach theorem, there exists $f_n\in X^*$ with $\|f_n\|_{X^*}=1$ vanishing on $x_1,\ldots,x_n$. Fix $z^*\in X^*$ and $y_0\in Y$ with $\|z^*\|_{X^*}=|y_0|_Y=1$, and set $$L_n(h_1,\ldots,h_i):=f_n(h_1)\prod_{t=2}^iz^*(h_t)\,y_0.$$ Then:
\begin{itemize}
\item \emph{Convergence in $L^p_\mu$.} Since $f_n(x_k)=0$ for $n\geq k$, we have $f_n(x_k)\to0$ for all $k$; since $\|f_n\|_{X^*}=1$ and $\{x_k\}$ is dense, $f_n(x)\to0$ for all $x\in X$.\footnote{Indeed, given $x\in X$ and $\varepsilon>0$, choose $k\in\mathbb N$ with $|x-x_k|_X<\varepsilon$; then, for all $n\geq k$, $|f_n(x)|\leq|f_n(x_k)|+|f_n(x-x_k)|\leq\|f_n\|_{X^*}|x-x_k|_X<\varepsilon$.} Hence $L_n\to0$ pointwise on $X^i$.  Since $|L_n(h_1,\ldots,h_i)|_Y^p\leq\prod_{t=1}^i|h_t|_X^p$, which is $\mu$-integrable by assumption, dominated convergence gives $\|L_n\|_{L^p_\mu(X^i,Y)}\to0$.
\item \emph{Operator norm.} $\|L_n\|_{\mathcal L^i(X,Y)}=\|f_n\|_{X^*}\|z^*\|_{X^*}^{i-1}|y_0|_Y=1$ for all $n\in\mathbb N$.
\end{itemize}

(ii) Let $Y$ be a Banach space and let $I\colon\big(\mathcal L^i(X,Y),\|\cdot\|_{\mathcal L^i(X,Y)}\big)\to\big(\mathcal L^i(X,Y),\|\cdot\|_{L^p_\mu(X^i,Y)}\big)$ be the identity map. By Proposition~\ref{lemma:multilinear}, $\big(\mathcal L^i(X,Y),\|\cdot\|_{\mathcal L^i(X,Y)}\big)$ is a Banach space, since $Y$ is; moreover, $I$ is a linear bijection, which is bounded by \eqref{eq:normr_bounded_opnorm}. If $X$ is finite dimensional, the two norms are equivalent by (i), so that $\big(\mathcal L^i(X,Y),\|\cdot\|_{L^p_\mu(X^i,Y)}\big)$ is complete. Conversely, assume that $\big(\mathcal L^i(X,Y),\|\cdot\|_{L^p_\mu(X^i,Y)}\big)$ is complete. Then both spaces are Banach spaces, so that, by the bounded inverse theorem \cite[Section~2.3]{brezis2011}, $I^{-1}$ is bounded, i.e.\ there exists $C>0$ such that
\[
\|L\|_{\mathcal L^i(X,Y)}\leq C\|L\|_{L^p_\mu(X^i,Y)}\qquad\text{for all }L\in \mathcal L^i(X,Y).
\]
Together with \eqref{eq:normr_bounded_opnorm}, this shows that the two norms are equivalent, so that $X$ is finite dimensional by (i).
\end{proof}

\begin{proposition}\label{prop:bounded_sets}
Let Assumption~\ref{ass:moments} hold. Assume that $\|\cdot\|_{L^p_\mu(X^i,Y)}$ is a norm on $\mathcal L^i(X,Y)$. On every operator-norm-bounded subset of $\mathcal L^i(X,Y)$, the topology induced by this norm coincides with the strong operator topology and the compact-open topology\footnote{topologies  of pointwise and uniform convergence on compact subsets of $X^i$, respectively}. In particular, all such norms, independently of the measure and the finite exponent, induce the same topology on these subsets. If $Y$ is finite dimensional, these topologies also coincide with the weak operator topology.
\end{proposition}
\begin{proof}
If $Y$ is finite dimensional, the weak and norm topologies on $Y$ coincide, so that the weak and strong operator topologies coincide. By Proposition~\ref{prop:Lp_norm_separable}, $X$ is separable; fix a countable dense set $D\subset X$. Since the strong operator and compact-open topologies coincide on operator-norm-bounded subsets and are metrizable (the pointwise topology is determined by the countably many evaluations on $D^i$), it suffices to show that, for every $(L_n)_{n\in\mathbb N}\subset\mathcal L^i(X,Y)$ with $\sup_n\|L_n\|_{\mathcal L^i(X,Y)}<\infty$ and every $L\in\mathcal L^i(X,Y)$,
\[
\|L_n-L\|_{L^p_\mu(X^i,Y)}\to0\quad\Longleftrightarrow\quad L_n(h)\to L(h)\quad\text{for all }h\in X^i.
\]
Indeed, set $M:=\sup_n\|L_n-L\|_{\mathcal L^i(X,Y)}<\infty$.

($\Leftarrow$) Let $L_n\to L$ pointwise. Since\footnote{Note that, if $X$ is Banach, this implication does not require an a priori uniform bound on $\|L_n-L\|_{\mathcal L^i(X,Y)}$, since this follows from pointwise convergence by the multilinear Banach--Steinhaus theorem.} \(
|(L_n-L)(h_1,\ldots,h_i)|_Y^p\leq M^p\prod_{t=1}^i|h_t|_X^p
\), dominated convergence gives $\|L_n-L\|_{L^p_\mu(X^i,Y)}\to0$.

($\Rightarrow$) Assume that $\|L_n-L\|_{L^p_\mu(X^i,Y)}\to0$. We first show that the weak$^*$ topology (hence, pointwise topology) and the scalar $L^p_\mu(X^i,\mathbb R)$-norm topologies coincide on the closed ball $\mathcal B_M$ of radius $M$ in $\mathcal L^i(X,\mathbb R)$. Indeed,
 let $E:=X^{\widehat\otimes_\pi i}$ be the completed $i$-fold projective tensor product, giving the isometric isomorphism $\mathcal L^i(X,\mathbb R)\cong E^*$ \cite[Section~2.2]{ryan2002tensor}. Thus $\mathcal B_M$, 
identifies with the closed ball of radius $M$ in $E^*$, which is weak$^*$-compact by Banach--Alaoglu. By the scalar version\footnote{Notice that the scalar seminorm is a norm by Remark~\ref{rem:target_nondegeneracy}.} of ($\Leftarrow$) and metrizability, the identity from $\mathcal B_M$ (endowed with this topology) into the scalar $L^p_\mu$-normed space is continuous. Thus this identity is a homeomorphism onto its image (i.e. a continuous bijection with continuous inverse), since its domain is compact and its target is Hausdorff\footnote{Note that this step may fail if the scalar $L^p_\mu$-seminorm (and hence the full $L^p_\mu$-seminorm) is not a norm.} \cite[Theorem~26.6]{munkres2000topology}; that is the pointwise and scalar $L^p_\mu$-norm topologies coincide on $\mathcal B_M$.

Consequently, for every $h\in X^i$ and $\varepsilon>0$, there exists $\delta>0$ such that if
\(
z\in\mathcal B_M,
\|z\|_{L^p_\mu(X^i,\mathbb R)}<\delta\), then  \(|z(h)|<\varepsilon.
\)
Applying this, for sufficiently large $n$ and every $y^*\in Y^*$ with $\|y^*\|_{Y^*}\leq1$, we have $\|y^*\circ(L_n-L)\|_{\mathcal L^i(X,\mathbb R)}\leq M$ and
\[
\begin{gathered}
\|y^*\circ(L_n-L)\|_{L^p_\mu(X^i,\mathbb R)}
\leq\|L_n-L\|_{L^p_\mu(X^i,Y)}<\delta
\Longrightarrow\quad |y^*((L_n-L)(h))|<\varepsilon.
\end{gathered}
\]

By the corollary of the Hahn--Banach theorem \cite[Corollary~1.4]{brezis2011}, for each $n$ there exists $y_n^*\in Y^*$ with $\|y_n^*\|_{Y^*}\leq1$ and $|y_n^*((L_n-L)(h))|=y:=|(L_n-L)(h)|_Y$. Hence, for all sufficiently large $n$,
\[
|(L_n-L)(h)|_Y=y=|y_n^*((L_n-L)(h))|<\varepsilon,
\]

which proves pointwise convergence.

\end{proof}

\section{Classical Examples}\label{sec:examples}
In this section, we discuss examples of the norms/seminorms in \eqref{norm:Lp_operators}, their relation to Hilbert--Schmidt and $\gamma$-radonifying operators on the associated RKHS.

\subsection{Discrete measures}\label{subsec:discrete}
\begin{example}[Discrete measures]\label{ex:discrete}
Let $\{x_n\}_{n\in\mathbb N}$ be a sequence in $\{h\in X:|h|_X\leq1\}^i$ whose closure contains $(S^{X,1})^i$, and let $w_n>0$, $n\in\mathbb N$, with $\sum_{n\in\mathbb N}w_n<\infty$. Then $\mu:=\sum_{n\in\mathbb N}w_n\delta_{x_n}$ is a finite Radon measure on $X^i$ satisfying Assumption~\ref{ass:moments}, with $\operatorname{supp}\mu\supset(S^{X,1})^i$, and \eqref{norm:Lp_operators} becomes a weighted $\ell^p$-norm of the values of $L$ along $\{x_n\}$:
\[
\|L\|_{L^p_\mu(X^i,Y)}=\Big(\sum_{n\in\mathbb N}w_n|L(x_n)|_Y^p\Big)^{1/p},\qquad L\in\mathcal L^i(X,Y),
\]
which is a norm by Proposition~\ref{lem:seminorm_norm}. The measure $\mu$ in the proof of Proposition~\ref{prop:Lp_norm_separable} is of this form. For $i=p=1$, $w_n=2^{-n}$ and $\{x_n\}$ dense in the closed unit ball of $X$, we obtain $\sum_{n\in\mathbb N}2^{-n}|Lx_n|_Y$, which is classically used to metrize the strong operator topology on bounded subsets of $\mathcal L(X,Y)$, see e.g.\ \cite[Remark~3.19(d)]{kruse-schwenninger2022}.
\end{example}

\subsection{\texorpdfstring{Hilbert--Schmidt and $\gamma$-radonifying operators}{Hilbert--Schmidt and gamma-radonifying operators}}\label{sec:HS}
Let $\mu$ be as in Proposition~\ref{prop:finite-dim}. By Proposition~\ref{prop:finite-dim}(ii), if $Y$ is a Banach space and $X$ is infinite dimensional, then $\big(\mathcal L^i(X,Y),\|\cdot\|_{L^p_\mu(X^i,Y)}\big)$ is not complete. However, in the following important examples, its completion can be identified explicitly (Theorem~\ref{thm:completion}).

For $p=2$ and $Y$ {a} separable Hilbert space, we have the following characterization:
\begin{example}[Hilbert--Schmidt norms on Reproducing Kernel Hilbert Spaces]\label{ex:HS_norms_RKHS}
{Let $X$ be a separable Banach space, let $Y$ be a
 separable
Hilbert space, and let $\nu$ be a (Borel) probability measure on
$X$ (hence, Radon) with $\int_X| h|_X^2\,d\nu(h)<\infty$. For simplicity, assume that $\nu$ is centered (i.e. with zero mean).
Denote the   covariance kernel by $K_\nu :X^*\times X^* \to \mathbb R$ and the covariance operator by $R_\nu:X^*\to X$,  characterized 
by  \cite[Section III.2.1]{vakhania-tarialadze}
\[
\langle y^*,R_\nu x^*\rangle_{X^*,X}=\int_X\langle x^*,h\rangle_{X^*,X}\langle y^*,h\rangle_{X^*,X}\,d\nu(h)=K_\nu(x^*,y^*).
\]
We define the associated  Reproducing Kernel Hilbert Space (RKHS) $X_\nu$ (equivalently identified with the Hilbert space associated with $R_\nu$ via $R_\nu x^* \longleftrightarrow K_\nu\left(x^*, \cdot\right)$)  as the completion of $R_\nu(X^*)$ under the scalar product $\langle R_\nu y^*,R_\nu x^*\rangle_{X_\nu}:=\langle y^*,R_\nu x^*\rangle_{X^*,X}=K_\nu(x^*,y^*)$; the reproducing identity holds $\langle z,R_\nu x^*\rangle_{X_\nu}=\langle z,K_\nu(x^*,\cdot)\rangle_{X_\nu}=\langle x^*,z\rangle_{X^*,X}$, for all $z\in X_\nu,x^*\in X^*$; we denote by
$\iota_\nu:X_\nu\hookrightarrow X$ the canonical embedding (see \cite[Section III.1.2-3]{vakhania-tarialadze}). Then
$R_\nu=\iota_\nu\iota_\nu^*$. Since $X$ is separable, $X_\nu$ is separable; let {us choose} an
orthonormal basis $\{e_n^\nu\}$. The product measure $\mu:=\otimes_{j=1}^i\nu$ satisfies Assumption~\ref{ass:moments} with $p=2$. Then, denoting the restriction of $L\in \mathcal L^i(X,Y)$ to $X^i_\nu$ by  $L_\nu :=L\circ(\iota_\nu,\ldots,\iota_\nu)$, we have
\begin{align}\label{eq:norm2=HS}
    \|L\|_{L^2_{\otimes_{j=1}^i\nu}(X^i,Y)}=\left\| L_\nu\right\|_{\mathcal L_2^i(X_\nu,Y)},\quad \forall  L\in \mathcal L^i(X,Y).
\end{align} }

{Indeed, let $i=1$. For
$L\in\mathcal{L}(X;Y)$, set $\lambda=\lambda_L:=L_{\#}\nu$. Then $\lambda$ is
{a} centered finite measure on $Y$ with finite second moment. By \cite[Section III.2.1 (e)]{vakhania-tarialadze}, we have $R_\lambda=LR_\nu L^*=L\iota_\nu(L\iota_\nu)^*$; by \cite[Section III.2.3, Remark]{vakhania-tarialadze} (using Riesz identification $Y^*\sim Y$), 
\begin{align*}
 \|L\|_{L^2_\nu(X,Y)}^2=\int_X| Lh|_Y^2\,d\nu(h)
 &=\int_Y| y|_Y^2\,d\lambda(y)
 =\operatorname{Tr}_Y(R_\lambda)=\operatorname{Tr}_Y\!\bigl((L\iota_\nu)(L\iota_\nu)^*\bigr) =\lVert (L_\nu)^*\rVert_{\mathcal{L}_2(Y,X_\nu)}^2
 =\lVert L_\nu\rVert_{\mathcal{L}_2(X_\nu,Y)}^2.
\end{align*}
For $i\geq2$, assume that the assertion holds for $i-1$; define
$A:X\to \mathcal L_2^{i-1}(X_\nu,Y)$, 
$
A(h_1)(h_2,\ldots,h_i)
:=L(h_1,\iota_\nu h_2,\ldots,\iota_\nu h_i).
$
By the induction hypothesis,
\[
\begin{aligned}
\|A(h_1)\|_{\mathcal L_2^{i-1}(X_\nu,Y)}^2
&=\int_{X^{i-1}}
   |L(h_1,h_2,\ldots,h_i)|_Y^2\,
   d\nu(h_2)\cdots d\nu(h_i)\leq
\|L\|_{\mathcal L^i(X,Y)}^2 |h_1|_X^2
\left(\int_X |h|_X^2\,d\nu(h)\right)^{i-1}.
\end{aligned}
\]
Thus $A$ is well defined and bounded. Moreover, under the
isometric identification
$
\mathcal L_2\bigl(X_\nu,\mathcal L_2^{i-1}(X_\nu,Y)\bigr)=\mathcal L_2^i(X_\nu,Y),
$
we have $\left(A \iota_\nu\right)\left(h_1\right)\left(h_2, \ldots, h_i\right)=L_\nu\left(h_1, \ldots, h_i\right)$. Thus, applying the
case $i=1$ to $A$, with Hilbert target
$\mathcal L_2^{i-1}(X_\nu,Y)$, and the induction hypothesis, we have
\[
\begin{aligned}
\|L\|_{L^2_{\otimes_{j=1}^i\nu}(X^i,Y)}^2&=\int_{X}\int_{X^{i-1}}|L(h_1,\ldots,h_i)|_Y^2\,
   d\nu(h_1)d\nu(h_2)\cdots d\nu(h_i)\\
   &=\int_X
   \|A(h_1)\|_{\mathcal L_2^{i-1}(X_\nu,Y)}^2\,d\nu(h_1)=\|A\iota_\nu\|_{
   \mathcal L_2(X_\nu,\mathcal L_2^{i-1}(X_\nu,Y))}^2=\|L_\nu\|_{\mathcal L_2^i(X_\nu,Y)}^2.
\end{aligned}
\]}
\end{example}
The covariance trace identity above relies on the Hilbert structure of $Y$ \cite[Section~III.2.3, Remark]{vakhania-tarialadze}. For general Banach targets and centered Gaussian Radon measures, an analogous identity holds in terms of iterated $\gamma$-radonifying norms, as shown in the following example.
\begin{example}[Gaussian measures: $\gamma$-radonifying operator norms]\label{ex:gaussian_measures_radonifying}Let $X$, $Y$ be Banach spaces, and let $\gamma$ be a centered Gaussian
Radon measure on $X$ (for the definition, see \cite{bogachev1998gaussian}). Its RKHS $X_\gamma$ coincides with the
Cameron--Martin space (therefore,  we will use the same notation $X_\gamma$ to indicate both) and is separable. The product measure $\mu:=\otimes_{j=1}^i\gamma$ satisfies Assumption~\ref{ass:moments} with $p=2$. Then, for all $L\in\mathcal L^i(X;Y)$, we have $L_\gamma\in\gamma^i(X_\gamma,Y)$ {(recall the iterated spaces $\gamma^i(H,Y)=\gamma(H,\gamma^{i-1}(H,Y))$ from {Appendix~\ref{app:basics}})} with\footnote{If $Y$ is Hilbert, then
$\gamma^i(X_\gamma,Y)=\mathcal L_2^i(X_\gamma,Y)$ isometrically, so this
reduces to~\eqref{eq:norm2=HS}}
\[
  \|L\|_{L^2_{\otimes_{j=1}^i\gamma}(X^i,Y)}
  =\|L_\gamma\|_{\gamma^i(X_\gamma,Y)}.
\]

Indeed, let  $i=1$; define the  
 Gaussian measure $\lambda:=L_\#\gamma$ on $Y$.  Its covariance is
$R_\lambda=LR_\gamma L^*=TT^*$, where we have used that
$R_\gamma=\iota_\gamma\iota_\gamma^*$ and we have set $T:=L\iota_\gamma$. Since $\iota_\gamma \in \gamma(X_\gamma,X)$ (see \cite[Proposition 8.6]{van-neerven2010}), by \cite[Theorem 6.2]{van-neerven2010}, we have $T\in\gamma(X_\gamma,Y)$.
Thus, by  \cite[Theorem~7.4]{van-neerven2010}, we have 
\[
  \|T\|_{\gamma(X_\gamma,Y)}^2
  =\int_Y|y|_Y^2d\lambda(y)
  =\int_X|Lh|_Y^2d\gamma(h)
  =\|L\|_{L^2_\gamma(X,Y)}^2.
\]
For $i\geq2$, we proceed by induction similarly to Example \ref{ex:HS_norms_RKHS} to prove the claim.
\end{example}

\begin{theorem}\label{thm:completion}
Let $i\geq1$, and consider one of the following settings:
\begin{enumerate}[(i)]
\item $X,Y,\nu$ are as in Example~\ref{ex:HS_norms_RKHS}, $\mu:=\otimes_{j=1}^i\nu$, $H:=X_\nu$, $\iota:=\iota_\nu$, and $\mathcal S:=\mathcal L_2^i(X_\nu,Y)$;
\item $X,Y,\gamma$ are as in Example~\ref{ex:gaussian_measures_radonifying}, $\mu:=\otimes_{j=1}^i\gamma$, $H:=X_\gamma$, $\iota:=\iota_\gamma$, and $\mathcal S:=\gamma^i(X_\gamma,Y)$.
\end{enumerate}
In both settings, Assumption~\ref{ass:moments} holds with $p=2$.
Let $\mathcal N:=\{L\in\mathcal L^i(X,Y):\|L\|_{L^2_\mu(X^i,Y)}=0\}$. Then:
\begin{enumerate}[label=(\alph*)]
\item the map $L\mapsto L\circ(\iota,\ldots,\iota)$ induces an isometric isomorphism from the completion of $\big(\mathcal L^i(X,Y)/\mathcal N,\|\cdot\|_{L^2_\mu(X^i,Y)}\big)$ onto $\mathcal S$;
\item if $Y\neq\{0\}$, then $\mathcal N=\{0\}$ if and only if $\iota(H)$ is dense in $X$.
\end{enumerate}
\end{theorem}
\begin{proof}
(a) By Examples~\ref{ex:HS_norms_RKHS} and~\ref{ex:gaussian_measures_radonifying}, $\Phi:\mathcal L^i(X,Y)\to \mathcal S$, $\Phi(L):=L\circ(\iota,\ldots,\iota)$, is a linear map satisfying
\begin{equation}\label{eq:Phi_isometry}
\|\Phi(L)\|_{\mathcal S}=\|L\|_{L^2_\mu(X^i,Y)},\qquad L\in\mathcal L^i(X,Y).
\end{equation}
Hence $\ker\Phi=\mathcal N$ and, with $E:=\big(\mathcal L^i(X,Y)/\mathcal N,\|\cdot\|_{L^2_\mu(X^i,Y)}\big)$, the map $\tilde\Phi\colon E\to\mathcal S$, $\tilde\Phi([L]):=\Phi(L)$, is a well-defined linear isometry.  Since $\mathcal S$ is complete, $\tilde\Phi$ extends to a linear isometry  $\widehat\Phi\colon\widehat E\to\mathcal S$ on the completion $\widehat E$, whose range is
\[
\widehat\Phi(\widehat E)=\overline{\Phi(\mathcal L^i(X,Y))}^{\,\mathcal S}\subset\mathcal S.
\]

Therefore, it remains to show that $\overline{\Phi(\mathcal L^i(X,Y))}^{\,\mathcal S}\supset\mathcal S.$ To this end, let $\mathcal T:=\operatorname{span}\{h_1\otimes\cdots\otimes h_i\otimes y:\ h_1,\ldots,h_i\in H,\ y\in Y\}$ (recall Appendix~\ref{app:basics}). For $x_1^*,\ldots,x_i^*\in X^*$ and $y\in Y$, we have $\Phi(x_1^*\otimes\cdots\otimes x_i^*\otimes y)=\iota^*x_1^*\otimes\cdots\otimes\iota^*x_i^*\otimes y$; i.e., $\Phi(\mathcal L^i(X,Y))$ contains all $h_1\otimes\cdots\otimes h_i\otimes y$ with $h_1,\ldots,h_i\in\iota^*(X^*)$ and $y\in Y$. Moreover, by \eqref{eq:cross_norm} (see below) and Proposition~\ref{lemma:multilinear}, the multilinear map $(h_1,\ldots,h_i)\mapsto h_1\otimes\cdots\otimes h_i\otimes y$ is continuous from $H^i$ into $\mathcal S$. Since $\iota^*(X^*)$ is dense\footnote{Indeed, $\iota$ maps $\iota^*(X^*)$ onto $R(X^*)$, where $R:=\iota\iota^*$ is the covariance operator $R_\nu$, resp.\ $R_\gamma$, and $H$ is, by construction, the completion of $R(X^*)$; see Examples~\ref{ex:HS_norms_RKHS} and~\ref{ex:gaussian_measures_radonifying}.} in $H$, it follows that $h_1\otimes\cdots\otimes h_i\otimes y\in\overline{\Phi(\mathcal L^i(X,Y))}^{\,\mathcal S}\supset\mathcal T$ for all $h_1,\ldots,h_i\in H$ and $y\in Y$. Thus, since $\mathcal T$ is dense in $\mathcal S$ (see below), we have
\[
\overline{\Phi(\mathcal L^i(X,Y))}^{\,\mathcal S}\supset\overline{\mathcal T}^{\,\mathcal S}=\mathcal S,
\]
so that $\widehat\Phi\colon\widehat E\to\mathcal S$ is an isometric isomorphism.

It remains to verify the two facts used above, namely that $\mathcal T$ is dense in $\mathcal S$ and that
\begin{equation}\label{eq:cross_norm}
\|h_1\otimes\cdots\otimes h_i\otimes y\|_{\mathcal S}=\prod_{t=1}^i|h_t|_H\,|y|_Y.
\end{equation}
In setting (i), write $\mathcal T_i:=\mathcal T$ and argue by induction on $i$, with $\mathcal L_2^0(H,Y):=Y$ and $\mathcal T_0:=Y$. Set $Z:=\mathcal L_2^{i-1}(H,Y)$, so that $\mathcal L_2^i(H,Y)=\mathcal L_2(H,Z)$ (see Appendix~\ref{app:basics}). For $h\in H$ and $z\in Z$, choosing an orthonormal basis $\{e_n\}$ of $H$, Parseval's identity gives
\begin{equation}\label{eq:rank_one_HS}
\|h\otimes z\|_{\mathcal L_2(H,Z)}^2=\sum_n\|(h\otimes z)e_n\|_Z^2=\sum_n|\langle e_n,h\rangle_H|^2\,\|z\|_Z^2=|h|_H^2\,\|z\|_Z^2.
\end{equation}
For $i=1$, i.e.\ $Z=Y$, this is \eqref{eq:cross_norm}; for $i\geq2$, applying \eqref{eq:rank_one_HS} with $h=h_1$ and $z=h_2\otimes\cdots\otimes h_i\otimes y\in Z$, and using \eqref{eq:cross_norm} for $i-1$, we obtain \eqref{eq:cross_norm} for $i$:
\[
\|h_1\otimes\cdots\otimes h_i\otimes y\|_{\mathcal L_2^i(H,Y)}=\|h_1\otimes(h_2\otimes\cdots\otimes h_i\otimes y)\|_{\mathcal L_2(H,Z)}=|h_1|_H\,\|h_2\otimes\cdots\otimes h_i\otimes y\|_{Z}=\prod_{t=1}^i|h_t|_H\,|y|_Y.
\]
Similarly, for $i=1$, $\mathcal T_1=H\otimes Y$ is the space of finite-rank operators from $H$ to $Y$, which is dense in $\mathcal L_2(H,Y)$. Next, let $i\geq2$ and assume that $\mathcal T_{i-1}$ is dense in $Z$. Given $A\in\mathcal L_2(H,Z)$ and $\varepsilon>0$, since the space of finite-rank operators $H\otimes Z$ is dense in $\mathcal L_2(H,Z)$, we can choose $\sum_{k=1}^m h_k\otimes z_k\in H\otimes Z$ with $\|A-\sum_{k=1}^m h_k\otimes z_k\|_{\mathcal L_2(H,Z)}<\varepsilon$ and, by the induction hypothesis, $z_k'\in\mathcal T_{i-1}$ with $|h_k|_H\|z_k-z_k'\|_Z<\varepsilon/m$. Then $\sum_{k=1}^m h_k\otimes z_k'\in H\otimes\mathcal T_{i-1}\subset\mathcal T_i$ and, by the triangle inequality and \eqref{eq:rank_one_HS},
\begin{align*}
\Big\|A-\sum_{k=1}^m h_k\otimes z_k'\Big\|_{\mathcal L_2(H,Z)}
&\leq\Big\|A-\sum_{k=1}^m h_k\otimes z_k\Big\|_{\mathcal L_2(H,Z)}+\sum_{k=1}^m\big\|h_k\otimes(z_k-z_k')\big\|_{\mathcal L_2(H,Z)}\\
&<\varepsilon+\sum_{k=1}^m|h_k|_H\|z_k-z_k'\|_Z<2\varepsilon,
\end{align*}%
which proves the density of $\mathcal T_i$ in $\mathcal L_2^i(H,Y)$.

In setting (ii), the same argument applies with $\mathcal L_2$ replaced by $\gamma$ (and $\gamma^0(H,Y):=Y$), since $\gamma^i(H,Y)=\gamma(H,\gamma^{i-1}(H,Y))$ (see Appendix~\ref{app:basics}) and since, by definition of $\gamma(H,Z)$, $H\otimes Z$ is dense in $\gamma(H,Z)$ and $\|h\otimes z\|_{\gamma(H,Z)}=|h|_H\|z\|_Z$ for $h\in H$ and $z\in Z$.

(b) Set $V:=\overline{\iota(H)}$. By \eqref{eq:Phi_isometry} and the injectivity of $\mathcal S\hookrightarrow\mathcal L^i(H,Y)$ (recall Appendix~\ref{app:basics}), $L\in\mathcal N$ if and only if $L$ vanishes on $\iota(H)^i$, or, by continuity, on $V^i$. If $V=X$, then $\mathcal N=\{0\}$. If $V\neq X$ and $Y\neq\{0\}$, choose $x^*\in X^*\setminus\{0\}$ vanishing on $V$ by the Hahn--Banach theorem, and $y_0\in Y\setminus\{0\}$. Then $x^*\otimes\cdots\otimes x^*\otimes y_0$ is a nonzero element of $\mathcal N$.
\end{proof}

\begin{remark}[Other finite exponents]\label{rem:gaussian_finite_p}
In the Gaussian setting of Theorem~\ref{thm:completion}(ii), the seminorms $\|\cdot\|_{L^p_\mu(X^i,Y)}$, $1\leq p<\infty$, are equivalent, with constants depending only on $i$ and the exponents \cite[Proposition~3.1]{maas2010} (see also \cite[Proposition~2.7(2)]{van-neerven2010}).  For every $1\leq p<\infty$, the restriction map in Theorem~\ref{thm:completion}(a) therefore extends to an isomorphism from the $L^p_\mu$-completion of $\mathcal L^i(X,Y)/\mathcal N$ onto $\gamma^i(X_\gamma,Y)$, with equivalent norms.
\end{remark}

\begin{remark}\label{rem:completion_unbounded}
In the setting of Theorem~\ref{thm:completion}, assume that $\mathcal N=\{0\}$ and that $X$ is infinite dimensional. By Proposition~\ref{prop:finite-dim}(ii), $\mathcal S$ contains elements $T$ whose pullbacks $T\circ(\iota^{-1},\ldots,\iota^{-1})$ are unbounded with respect to $|\cdot|_X$ and therefore admit no bounded multilinear extension to $X^i$. For instance, for centered Gaussian $\gamma$, $i=1$ and $Y=\mathbb R$, this corresponds to the classical fact that $\overline{X^*}^{\,L^2(X;\gamma)}$, identified with $(X_\gamma)^*$, consists  of (equivalence classes of) $\gamma$-measurable linear functionals, which need not admit continuous representatives on $X$ \cite[Theorems~2.10.9, 2.10.11]{bogachev1998gaussian}.
\end{remark}

\paragraph{Relation to \cite{defeo2026universal}.} 
A preliminary treatment of this material appeared in the appendices of versions v1--v4 of \cite{defeo2026universal}. However, space constraints limited the scope of the treatment. The author subsequently decided to develop the present article to  provide a systematic treatment.

\paragraph{AI disclosure.}In developing the preliminary material from \cite{defeo2026universal} into the present article, the author used ChatGPT (OpenAI) and Claude (Anthropic) to assist in preparing the manuscript, including developing mathematical arguments.
 The author takes full responsibility for the content of this article.

 \paragraph{Acknowledgments.} 
The author acknowledges funding by the Deutsche Forschungsgemeinschaft (DFG, German Research Foundation) – CRC/TRR 388 ``Rough Analysis, Stochastic Dynamics and Related Fields'' – Project ID 516748464.

\appendix
\section{Notation}\label{app:basics}
 \paragraph{Normed spaces.} Let $(X,|\cdot|_X)$ be a normed space. When there is no chance of confusion we simply write $|\cdot|$ for the norm. {We denote the sphere $S^{X,R}=\{x\in X:|x|_X= R\}, R>0$.}
For $k\in \mathbb N$, denote by $X^k$ the product space endowed with the norm $\left|\left(x_1, \ldots, x_k\right)\right|_{X^k}:=\left(\left|x_1\right|^2+\ldots+\left|x_k\right|^2\right)^{1 / 2}$. For a subset $A$ of a normed space $Z$, we denote by $\overline{A}^{\,Z}$, or simply by $\overline{A}$ when $Z$ is clear from the context, the closure of $A$ in $Z$, and, for a normed space $E$, we denote by $\widehat E$ its completion.

Let $(Y,|\cdot|_Y)$ be another normed space.
We denote by $\mathcal L(X,Y)$ the normed space of linear bounded  operators from $X$ to $Y$, endowed with the operator norm $\|L\|_{\mathcal L(X,Y)} :=\sup_{|x|_X=1}|Lx|_Y , $ $  L  \in \mathcal L(X,Y).$
We denote by $ X^*=\mathcal L(X,\mathbb R)$ the dual space, and by $\langle\cdot,\cdot\rangle_{X^*,X}$ the duality pairing.
When $Y=X$ we simply write $\mathcal{L}(X)=\mathcal{L}(X,X).$

We denote by $\mathcal L^i(X;Y)$ 
the  space of  continuous multilinear operators $T: X^i \rightarrow Y$. We recall the following standard result \cite[(A.2.5)]{Flett_1980}.
\begin{proposition}\label{lemma:multilinear}
    Let $T: X^i \rightarrow Y$ be multilinear, then the following are equivalent (i) $T$ is continuous; (ii) $T$ is continuous at one point; (iii) $\|T\|_{\mathcal L^i(X,Y)}=\sup \left\{\left|T\left(h_1, \ldots, h_i\right)\right |_Y:\left|h_1\right|_X = 1, \ldots,\left|h_i\right|_X = 1\right\}<\infty$; (iv) $\mathcal C=\left\{C \geq 0:\left|T\left(h_1, \ldots, h_i\right)\right|_Y \leq C \ \Pi_{j=1}^i|h_j|_X, \  \forall h_1, \ldots, h_i \in X \right\} \neq \varnothing$.
If one (hence, all) of these conditions {holds}, then $\|T\|_{\mathcal L^i(X,Y)} \in \mathcal C$ and $\|T\|_{\mathcal L^i(X,Y)}=\inf \mathcal C$. 
\end{proposition}
The space $\mathcal L^i(X;Y)$  is a normed space when endowed with the norm $\|\cdot\|_{\mathcal L^i(X,Y)}$, and a Banach space if $Y$ is complete. 
We use the notation $\mathcal{L}^i(X)$ when $X=Y$. 

For $x_1^*,\ldots,x_i^*\in X^*$ and $y\in Y$, we define $x_1^*\otimes\cdots\otimes x_i^*\otimes y\in\mathcal L^i(X,Y)$ by $(x_1^*\otimes\cdots\otimes x_i^*\otimes y)(h_1,\ldots,h_i):=\prod_{t=1}^i\langle x_t^*,h_t\rangle_{X^*,X}\,y$; accordingly, we identify the algebraic tensor product $X^*\otimes\cdots\otimes X^*\otimes Y$ ($i$ factors $X^*$) with a subspace of $\mathcal L^i(X,Y)$.

 \paragraph{Hilbert spaces.} If   $(X,\langle \cdot ,\cdot \rangle_X)$ is a Hilbert space, we denote the induced norm by   $ |\cdot|_X:=(\langle \cdot ,\cdot \rangle_X)^{1/2}.$ If $X$ is separable   with orthonormal basis $\{e_i^X\}$, we have  
 $x=\sum_{i\in \mathbb N}\langle x,e_i^X \rangle e_i, $ $ |x|^2=\sum_{i\in \mathbb N}|\langle x,e_i^X \rangle|^2.$
Throughout the whole paper we identify each Hilbert space $H$ with its dual $H^*.$ Accordingly, for a normed space $Y$, $i\geq1$, $h_1,\ldots,h_i\in H$ and $y\in Y$, $h_1\otimes\cdots\otimes h_i\otimes y\in\mathcal L^i(H,Y)$ is the map $(k_1,\ldots,k_i)\mapsto\prod_{t=1}^i\langle k_t,h_t\rangle_H\,y$ (see Appendix~\ref{app:basics}); this identifies the algebraic tensor product $H\otimes\cdots\otimes H\otimes Y$ ($i$ factors $H$) with a subspace of $\mathcal L^i(H,Y)$ and, for $i=1$, $H\otimes Y$ with the finite-rank operators from $H$ to $Y$, $(h\otimes y)h'=\langle h',h\rangle_H\,y$. 

Let   $(Y,\langle \cdot ,\cdot \rangle_Y)$ be another  Hilbert space.
If $X,Y$ are separable, we denote by $\mathcal{L}_1(X,Y)$ the Banach space of trace-class operators endowed with the norm  
$\|L\|_{\mathcal{L}_1(X,Y)}:=\inf \left \{\sum_{i \in \mathbb N} |a_i|_X |b_i|_Y : Lx=\sum_{i \in \mathbb N} b_i  \langle a_i,x \rangle, \ a_i \in X, b_i \in  Y , \forall i \in \mathbb N \right \} ,  $ $ L \in \mathcal{L}_1(X,Y).$ 
We denote by  $\mathcal{L}_2(X,Y)$ the Hilbert space of Hilbert-Schmidt operators from $X$ to $Y$. The scalar product in $\mathcal{L}_2(X,Y)$ and its induced norm are respectively given by
$\langle L,T \rangle_{\mathcal{L}_2(X,Y)}:= \sum_{k \in \mathbb N} \langle Le_k,Te_k \rangle_Y, $ $   \|L\|_{\mathcal{L}_2(X,Y)}:=\left (\sum_{k \in \mathbb N} |L  e_k|_Y^2 \right)^{1/2} $ $  L,T \in \mathcal{L}_2(X,Y),$
where  $\{ e_k\}$  is any orthonormal basis of $X$. We will consider the space of $i$-linear Hilbert-Schmidt operators $\mathcal L_2^i\left(X,Y\right)$, with inner product 
$
\langle L, T\rangle_{\mathcal L_2^i\left(X, Y\right)}=\sum_{j_1, \ldots, j_i\in \mathbb N}\left\langle L\left(e_{j_1}, \ldots, e_{j_i}\right), T\left(e_{j_1}, \ldots, e_{j_i}\right)\right\rangle_{Y},
$
and induced Hilbert-Schmidt norm $\|\cdot\|_{\mathcal{L}_2^i\left(X, Y\right)}$. The canonical isometric identification gives $
\mathcal L_2\bigl(X,\mathcal L_2^{i-1}(X,Y)\bigr)=\mathcal L_2^i(X,Y).
$
  When $Y=X$ we simply write $\mathcal{L}_1(X)=\mathcal{L}_1(X,X)$, $\mathcal{L}_2(X)=\mathcal{L}_2(X,X)$, $\mathcal{L}^i_2(X)=\mathcal{L}_2^i(X,X)$.
  We denote by $\mathcal{L}^+_1(X)\subset \mathcal{L}_1(X)$ the {cone} of positive operators. If $T\in \mathcal{L}_1(X)$ we can define its trace by  $\operatorname{Tr}(T):=\sum_{k \in \mathbb{N}}\left\langle T e_k, e_k\right\rangle_X$, where $\operatorname{Tr}(T)$ is independent of the   orthonormal basis $\{ e_k\}$.

\paragraph{Borel measures.} Let $X,Y$ be normed spaces. We say that a (positive) Borel measure $\mu$ on $X$ is finite, or, simply, $\mu$ is a finite measure on $X$ (dropping Borel for brevity), if $\mu(X)<\infty$; if $\mu(X)=1$ it is called  a Borel probability measure. {A measure $\mu$ on $X$ is called Radon if it is locally finite and inner regular, i.e.\ $\mu(B)=\sup\{\mu(C):C\subset B,\ C\text{ compact}\}$ for every $B\in\mathcal B(X)$; if $\mu$ is finite, this is equivalent to tightness,  i.e. for every $\epsilon>0$ there exists a compact set $C\subset X$ such that $\mu(X\setminus C)<\epsilon$. Every finite measure on a separable Banach space is Radon.}

 Let $T:X\to Y$ be Borel measurable. We define the pushforward measure $T_\#\mu$ on $Y$ by
$
(T_\#\mu)(A):=\mu\big(T^{-1}(A)\big),$ $ A\in \mathcal B(Y).$
Then, for every measurable function $\varphi:Y\to\mathbb R$,
$
\int_Y \varphi(y)\,d(T_\#\mu)(y)
=
\int_X \varphi\big(T(x)\big)\,d\mu(x).$ We define $\operatorname{supp}(\mu):=\left\{x \in X : \forall N_x :\left(x \in N_x \Rightarrow \mu\left(N_x\right)>0\right)\right\}$, where $N_x$ denotes a neighborhood of $x\in X.$

Given finite Radon measures $\nu_1,\ldots,\nu_i$ on a separable normed space $X$, we denote by $\nu_1\otimes\cdots\otimes\nu_i$ their product measure on $X^i$, which is defined on $\mathcal B(X)^{\otimes i}=\mathcal B(X^i)$, and we set $\otimes_{j=1}^i\nu:=\nu\otimes\cdots\otimes\nu$ ($i$ factors).

 \paragraph{$\gamma$-radonifying operators.}
{Let $H$ be a Hilbert space and $Y$ a Banach space. If $h_1,\ldots,h_N$ are orthonormal in $H$,
$y_1,\ldots,y_N\in Y$, and $(g_n)$ are independent standard Gaussian
variables, set
$\left\|\sum_{n=1}^N h_n\otimes y_n\right\|_{\gamma(H,Y)}^2
  :=\E\left|\sum_{n=1}^N g_ny_n\right|_Y^2.$
This norm is independent of the orthonormal representation. Following
\cite[Definition~3.7]{van-neerven2010}, $\gamma(H,Y)$ is the completion
of $H\otimes Y$ in this norm. The canonical inclusion into $\mathcal L(H,Y)$
extends to an injective contraction $\gamma(H,Y)\hookrightarrow\mathcal L(H,Y)$;
after this identification, $T\in\mathcal L(H,Y)$ is called $\gamma$-radonifying
if $T\in\gamma(H,Y)$ \cite[Definition~3.13]{van-neerven2010}.
For $i\geq1$, define the iterated spaces by
$\gamma^1(H,Y):=\gamma(H,Y)$ and, recursively,
$\gamma^i(H,Y)
  :=\gamma\bigl(H,\gamma^{i-1}(H,Y)\bigr),
 i\geq2.$ If $H,Y$ are {separable Hilbert spaces}, we have $\mathcal L_2^i(H,Y)=\gamma^i(H,Y)$.}
Recursively, we regard $\gamma^i(H,Y)$ as a subspace of $\mathcal L^i(H,Y)$ by $\gamma(H,\gamma^{i-1}(H,Y))\hookrightarrow\mathcal L(H,\gamma^{i-1}(H,Y))\hookrightarrow\mathcal L(H,\mathcal L^{i-1}(H,Y))=\mathcal L^i(H,Y)$.

\bibliography{refs}
\bibliographystyle{plain}
\end{document}